\documentclass[a4paper,10pt]{amsart}
\usepackage[utf8]{inputenc}
\usepackage[T1]{fontenc}
\usepackage{amsmath,amssymb,amsthm,amsfonts}
\usepackage{mathrsfs}
\usepackage{url}

\usepackage[
pagebackref = true,
colorlinks = true,
pdfstartview = FitH,
pdfpagemode = UseNone,
pdfauthor    = {Cherubini Laaksonen},
pdftitle     = {Title},
pdfkeywords  = {keywords},
pdfcreator   = {Pdflatex},
linktoc = none,
bookmarks = false
]
{hyperref}

\usepackage{mathtools}

\numberwithin{equation}{section}

\usepackage[margin=1.25in]{geometry}

\def\eps{\epsilon}

\def\II{\mathbb{I}}
\def\ZZ{\mathbb{Z}}
\def\RR{\mathbb{R}}
\def\SS{\mathbb{S}}
\def\TT{\mathbb{T}}
\def\NN{\mathbb{N}}

\def\Ccal{\mathcal{C}}
\def\Dcal{\mathcal{D}}
\def\Ecal{\mathcal{E}}
\def\Gcal{\mathcal{G}}
\def\Hcal{\mathcal{H}}
\def\Ncal{\mathcal{N}}
\def\Rcal{\mathcal{R}}
\def\Scal{\mathcal{S}}
\def\Vcal{\mathcal{V}}
\def\Xcal{\mathcal{X}}

\newcommand{\EE}{\mathbb{E}}

\newcommand\ringring[1]{%
  {% make an Ord atom
   \mathop{\kern0pt #1}\limits^{% set a box over the variable
     \vbox to-1.85ex{
       \kern-2ex % lower the ring accents
       \hbox to 0pt{\hss\normalfont\kern.1em \r{}\kern-.45em \r{}\hss}%
       \vss % fill
     }% end of \vbox
   }% end of the superscript
  }% end of \mathop
}

\newcommand{\simD}{\mathrel{\mathring{\sim}}}
\newcommand{\sime}{\mathrel{\ringring{\sim}}}
\newcommand{\simf}{\mathrel{\dot{\sim}}}
\newcommand{\simfe}{\mathrel{\ddot{\sim}}}

\DeclareMathOperator{\Var}{Var}
\DeclareMathOperator{\tr}{tr}
\DeclareMathOperator{\vol}{vol}
\DeclareMathOperator{\meas}{meas}
\DeclareMathOperator{\rk}{rk}

\theoremstyle{plain}
\newtheorem{theorem}{Theorem}[section]
\newtheorem{proposition}[theorem]{Proposition}
\newtheorem{lemma}[theorem]{Lemma}

\theoremstyle{definition}
\newtheorem{remark}{Remark}

\title{On the variance of the nodal volume of arithmetic random waves}

\author{Giacomo Cherubini}%Giacomo Cherubini
\author{Niko Laaksonen}%Niko Laaksonen

\address{
    Alfr\'ed R\'enyi Institute of Mathematics,
    POB 127, Budapest H-1364, Hungary;
}
\email{
    cherubini.giacomo@renyi.hu\\
    laaksonen.niko@renyi.hu
}

\subjclass[2010]{Primary 58J50; Secondary 58J37, 60G15, 60G60, 11P21}
\date{\today}

\begin{document}

%% TITLE AND ABSTRACT %%
\maketitle

%% CONTENT %%

\begin{abstract}
  Rudnick and Wigman (\emph{Ann.\ Henri Poincar\'{e}}, 2008) conjectured that the
  variance of the volume of
  the nodal set of arithmetic random waves on the $d$-dimensional torus is
  $O(E/\Ncal)$, as $E\to\infty$, where $E$ is the energy and $\Ncal$
  is the dimension of the eigenspace corresponding to $E$. Previous results
  have established this with
  stronger asymptotics when $d=2$ and $d=3$. In this brief note we prove
  an upper bound of the form $O(E/\Ncal^{1+\alpha(d)-\epsilon})$,
  for any $\epsilon>0$ and $d\geq 4$,
  where $\alpha(d)$ is positive and tends to zero with $d$.
  The power saving is the best possible with the current
  method (up to $\epsilon$)
  when $d\geq 5$ due to the proof of the $\ell^{2}$-decoupling conjecture by
  Bourgain and Demeter.
\end{abstract}

\section{Introduction}

Let $d\geq 4$ and $m\geq 1$. Denote by $F$
a Gaussian random eigenwave on the $d$-dimensional torus
$\TT^{d}=\RR^{d}/\ZZ^{d}$ of eigenvalue $E=4\pi^2m$, that is
\begin{equation}\label{star}
F(x) = \frac{1}{\sqrt{\Ncal}} \sum_{\mu\in\Ecal} a_\mu e(\mu\cdot x),
\end{equation}
where
\[
\Ecal = \Ecal_m = \{\mu\in\ZZ^d:\; |\mu|^2=m\}, \quad \Ncal=|\Ecal|,
\]
and the coefficients $a_\mu$ are complex standard Gaussian random variables.
We assume that the $a_{\mu}$ are independent except for the relation
$a_{-\mu}=\overline{a_\mu}$, which makes $F$ real-valued. Such $F$ is called
(after~\cite{krishnapur_nodal_2013}) an \emph{arithmetic random wave}.

The function $F$ serves as a model for a generic eigenfunction of $\TT^{d}$.
Due to the arithmetic structure of
the set of frequencies $\Ecal$,
the properties of $F$ have been extensively studied by both analysts and number
theorists~\cite{krishnapur_nodal_2013, bourgain_toral_2014,
rudnick_volume_2008}.
Of particular interest to us are the nodal set of $F$, defined as
\[
V = \{x\in\TT^d:\; F(x)=0\},
\]
and its volume $\Vcal=\meas(V)$ (with respect to the $(d-1)$-dimensional
Lebesgue measure).
Rudnick and Wigman~\cite[Proposition~4.1]{rudnick_volume_2008} computed the
expected value of the nodal volume,
\begin{equation}\label{starstar}
\EE[\Vcal] = \Gcal_d \cdot \sqrt{\frac{m}{d}},
\qquad
\Gcal_d
=
\sqrt{4\pi} \cdot \frac{\Gamma(\frac{d+1}{2})}{\Gamma(\frac{d}{2})},
\end{equation}
which is in accordance with Yau's conjecture $\Vcal\asymp \sqrt{E}$.
They also proved~\cite[Theorem~6.1]{rudnick_volume_2008} that for all dimensions
$d\geq 2$, we have the upper bound
\begin{equation}\label{eq:rudwig}
    \Var(\Vcal) \ll \frac{m}{\sqrt{\Ncal}}.
\end{equation}
Furthermore, in the same paper it was conjectured
(see~\cite[p.~110]{rudnick_volume_2008},
also~\cite[(6)]{krishnapur_nodal_2013}) that
the stronger upper bound
\begin{equation}\label{eq:rudwigconj}
\Var(\Vcal) \ll \frac{m}{\Ncal}
\end{equation}
should hold.
In a landmark paper of Krishnapur, Kurlberg and
Wigman~\cite{krishnapur_nodal_2013} the authors showed that for $d=2$ the
correct asymptotic is in fact of the order $c_{2}E/\Ncal^{2}$ (where
$c_{2}=\Gcal_{2}^{2}/128$ in the generic case) since the coefficients contributing
to the leading order term cancel out perfectly. They called this behaviour
\emph{arithmetic Berry cancellation}, see~\cite[\S1.6]{krishnapur_nodal_2013}
for more details.
This was later extended to $d=3$ by Benatar and
Maffucci~\cite{benatar_random_2017} who proved an asymptotic of the form
$c_{3}E/\Ncal^{2}$ with $c_{3}=(2/375)\Gcal_{3}^{2}$. Notice that the
leading term is again smaller than expected due to the same cancellation phenomenon.

In this note we prove that the conjecture~\eqref{eq:rudwigconj} of Rudnick and Wigman
holds also for $d\geq4$ with a power saving that tends to 0 as $d\to\infty$.

\begin{theorem}\label{intro:thm}
Let $\Vcal$ denote the nodal volume of arithmetic random waves
of eigenvalue $4\pi^{2}m$ on $\TT^d$ for $d\geq 4$ and assume that $m$
is odd if $d=4$. Then
\[
    \Var(\Vcal) \ll m \Ncal^{-1-\alpha(d)+\epsilon},
\]
where
\begin{equation}\label{eq:alphad}
    \alpha(d)=
    \begin{cases}
        \displaystyle\frac{2}{d-1} & \text{if } d=4,\\
        \displaystyle\frac{2}{d-2} & \text{if } d\geq 5.
    \end{cases}
\end{equation}
\end{theorem}

An important part of the proof concerns the structure of the frequency set
$\Ecal_{m}$. In particular, we need information on the number of ways that
elements of $\Ecal_{m}$ can sum up to zero. Therefore, we define the set of
(linear) $\ell$-correlations of $\Ecal_{m}$ as
\[
\Ccal(\ell) = \Ccal_m^{(d)}(\ell)
=
\bigl\{(\mu_1,\ldots,\mu_\ell)\in\Ecal^\ell: \mu_1+\cdots+\mu_\ell=0\bigr\},
\]
and the set of non-degenerate correlations by
\[
\Xcal(\ell) = \Xcal_m^{(d)}(\ell)
=
\Bigl\{
(\mu_1,\ldots,\mu_\ell)\in \Ccal_m^{(d)}(\ell):
\forall\,\Hcal\subsetneq\{1,\ldots,\ell\},\;\sum_{i\in\Hcal}\mu_i\neq 0
\Bigr\}.
\]
The main theorem then follows from the following proposition together with
existing bounds for $\Ccal(4)$.
Such bounds were proved by
Bourgain and Demeter for $d=4$ in~\cite{bourgain_new_2015} and
for $d\geq5$ in~\cite{bourgain2015_annals} as part of their proof of the
$\ell^{2}$-decoupling conjecture.

\begin{proposition}\label{intro:proposition}
    Let $d$ and $m$ be as in Theorem~\ref{intro:thm}. We then have
    \begin{equation}\label{eq:propeq}
        \Var(\Vcal) = \frac{m}{\Ncal^2}
        \left(
            \frac{(d-1)}{d(d+2)^3}\Gcal_d^2
            +
            O\left(
                \frac{1}{m^{(d-3)/4-\eps}}
                +
                \frac{|\Xcal(4)|}{\Ncal^2}
                +
                \frac{|\Ccal(6)|}{\Ncal^4}
            \right)
        \right),
    \end{equation}
    where $\Gcal_d$ is given in \eqref{starstar}.
\end{proposition}
Notice that when $d=3$, the constant in~\eqref{eq:propeq} agrees
with~\cite[Theorem~1.2]{benatar_random_2017}, and when $d=2$ it agrees
with~\cite[Theorem~1.1]{krishnapur_nodal_2013} (in the case when
$\widehat{\mu}(4)\to0$ in their notation, i.e.~when the sequence of lattice
points equidistributes, which happens for a generic sequence).

We stress that the results in~\cite{krishnapur_nodal_2013}
and~\cite{benatar_random_2017} are stronger, because they obtain
true asymptotics for the variance $\Var(\Vcal)$.
In two and three dimensions the limiting distribution of $\Vcal$
has also been determined, see \cite{marinucci_nonuniversality_2016}
and \cite{cammarota_nodal_2019}.
On the other hand, when $d\geq 4$,
by a straightforward analytic argument
(see~\cite[p.~3050]{benatar_random_2017}), it
is possible to show the lower bound
\begin{equation}\label{intro:lowerbounds}
  |\Ccal_{m}^{(d)}(\ell)|\gg \Ncal^{\ell-d/(d-2)} = \Ncal^{\ell-1-2/(d-2)},
\end{equation}
for any $\ell\geq 4$.
In light of Proposition~\ref{intro:proposition} and the fact that
$|\Ccal(4)|$ and $|\Xcal(4)|$ are roughly of the same order
(see section~\ref{sec:lin_corrs}),
\eqref{intro:lowerbounds} means that the upper bound in Theorem \ref{intro:thm}
is the best possible with the current method
for all $d\geq 5$ up to the factor $\Ncal^{\epsilon}$.

Moreover, our proof suggests that in order to obtain an asymptotic formula in
higher dimensions one would require better control on the singular set
(see section~\ref{S3:singularset}) as well as precise information about
$\ell$-correlations of arbitrary length. In particular, it would be crucial to be able
to detect cancellation over non-degenerate correlations for all $\ell\geq 4$,
which seems very difficult. This is in contrast to the situation for $d=2,3$,
where only $|\Xcal(4)|$ is essential.

\subsection*{Acknowledgements}
We thank Mikl\'os Ab\'ert for making us interested in this problem,
and Gergely Harcos for useful discussions.
We also thank Riccardo Maffucci for his comments on an earlier
version of the paper. This work was supported by ERC Consolidator Grant 648017
and by the R\'enyi Int\'ezet Lend\"ulet Automorphic Research Group.

%%%%%%%%%%%%%%%
%% SECTION 2 %%

\section{Lattice Points on Spheres and Linear Correlations}

Recall that $\Ecal$ denotes the set of lattice points on the sphere
of radius $\sqrt{m}$ in $\RR^d$ and $\Ncal$ is the cardinality of $\Ecal$.
It is well-known (see e.g.~\cite[Corollary~11.3]{iwaniec1997})
that, for $d\geq 5$,
\begin{equation}\label{eq:Nsize}
    \Ncal \asymp m^{d/2-1}.
\end{equation}
The above estimate holds also for $d=4$ with an additional factor $m^{o(1)}$
if we restrict to $m$ being odd, which
we shall assume implicitly for the rest of the paper (for $d=4$ only).
When $m\to\infty$ the normalised lattice points $\mu/|\mu|$
become equidistributed on the unit sphere.
This is known with an explicit power saving in the error term.

\begin{lemma}\label{lemma:equid}
Let $g$ be a smooth function on the unit sphere $\SS$ in $\RR^d$ for $d\geq 4$.
Then, for every $\eps>0$ and sufficiently large $s\in\NN$, we have
\begin{equation}\label{eq:equid}
    \frac{1}{\Ncal} \sum_{\mu\in\Ecal} g\left(\frac{\mu}{|\mu|}\right)
    =
    \frac{1}{\mathrm{vol}(\SS)}\int_{\SS} g(x)\, dx
    +
    O\left((\lVert \Delta^{s}g\rVert_{\infty}+1) m^{-(d-3)/4+\eps}\right),
\end{equation}
as $m\to\infty$. The implied constant depends only on $d$, $s$ and $\epsilon$.
\end{lemma}
\begin{proof}
  The exponent in the error term corresponds to the classical bound
  $a_{n}\ll n^{k/2-1/4+\epsilon}$ for the Fourier coefficients of a holomorphic
  cusp form of weight $k$ on $\Gamma_{0}(N)$, see for example
  \cite[Proposition~11.4]{iwaniec1997}
  where~\eqref{eq:equid} is proved for a fixed $g$.
  We need to be slightly more precise since it is crucial for us to control the dependence on $g$.
  This is certainly possible by combining the method of Palczewski et
  al.~\cite[Lemmas~6--7]{palczewski} and Fomenko~\cite{fomenko}. Their proofs
  would imply much stronger error terms (depending on the parity of $d$) than what we state, but this is not crucial for our argument.
  In fact, any rate of decay is enough so in principle even
  the earlier results of Pommerenke~\cite{pommerenke,pommerenkeB} and
  Malyshev~\cite{malyshev} (see also~\cite{linnik1968}) would suffice
  provided that one could make the dependence on $g$ explicit.
\end{proof}
Note that in~\eqref{eq:equid} the factor $dx/\vol(\SS)$ is the invariant measure on $\SS$.
We apply Lemma~\ref{lemma:equid} to evaluate the $k$-th moment
of the normalised inner product of two lattice points
\begin{equation}\label{Bk:def}
B_k := \frac{1}{m^k\Ncal^2} \sum_{\mu_1,\mu_2\in\Ecal} (\mu_1\cdot\mu_2)^k.
\end{equation}
\begin{lemma}\label{lemma:bk}
    Let $k\geq 1$ and let $B_k$ be as in \eqref{Bk:def}. Then
    \begin{equation}\label{eq:bk}
        B_k = \begin{cases}
            0 &\text{if $k$ is odd,}\\
            \frac{1}{d} &\text{if $k=2$,}\\
            \frac{\Gamma(\frac{k+1}{2})\Gamma(\frac{d}{2})}
            {\Gamma(\frac{k+d}{2})\Gamma(\frac{1}{2})} + O(m^{-(d-3)/4+\epsilon})
                        &\text{if $k\geq 4$ and $k$ is even.}
        \end{cases}
    \end{equation}
\end{lemma}
Notice that the leading term in $B_{k}$ is in fact always rational.
\begin{proof}
    The first case is immediate since the term corresponding to
    $(\mu_{1}, \mu_{2})$ cancels out with e.g.\ $(-\mu_{1}, \mu_{2})$.
    The second case follows from~\cite[Lemma~2.3]{rudnick_volume_2008}.
    Now, if $k\geq 4$ and even, we generalise the argument
    of~\cite[Lemma~2.5]{benatar_random_2017} and write
    \[B_{k}=\frac{1}{\Ncal^{2}}\sum_{\mu_{1},\mu_{2}}\cos^{k}\varphi_{\mu_{1},\mu_{2}},\]
    where $\varphi_{\mu_{1},\mu_{2}}$ is the angle between $\mu_{1}$ and
    $\mu_{2}$. We first consider the sum over $\mu_{2}$.
    From Lemma~\ref{lemma:equid} it follows that
    \begin{equation}\label{eq:bkint}
        \frac{1}{\Ncal}\sum_{\mu_{2}}\cos^{k}\varphi_{\mu_{1},\mu_{2}}
        =\frac{1}{\vol
        \SS}\int_{\SS}\cos^{k}\varphi_{\mu_{1},x}\,dx+O(m^{-(d-3)/4+\epsilon}),
    \end{equation}
    since $||\Delta^{s} \cos^{k}\varphi_{\mu_{1},x}||_{\infty}$ is bounded.
    Notice that the integral is independent of $\mu_{1}$. In order to evaluate
    it, we let $\tilde{\mu}=(1,0,\ldots,0)$ and introduce the $d$-dimensional
    spherical coordinates
    $x=(x_{1},\ldots,x_{d})\to(1,\theta,\phi_{1},\ldots,\phi_{d-2})$ by
    \begin{gather*}
        x_{i} = \cos\phi_{i}\prod_{j=1}^{i-1}\sin\phi_{j}, \quad\text{for
        $i=1,\ldots,d-2$,}\\
        x_{d-1} = \cos\theta\prod_{j=1}^{d-2}\sin\phi_{j},\quad
            x_{d} = \sin\theta\prod_{j=1}^{d-2}\sin\phi_{j},
    \end{gather*}
    where $\phi_{i}\in[0,\pi)$ and $\theta\in[0,2\pi)$ and for $i=1$ the empty
    product is understood as having the value 1. The Jacobian of
    this transformation is
    \[J_{d}=\prod_{j=1}^{d-2}\sin^{d-1-j}\phi_{j}.\]
    Applying this transformation to the integral in~\eqref{eq:bkint} we get
    \begin{align*}
        \int_{\SS}\cos^{k}\varphi_{\mu_{1},x}\,dx
        &= \int_{\SS}\cos^{k}\varphi_{\tilde{\mu},x}\,dx\\
        &= \int_{0}^{2\pi}d\theta \int_{0}^{\pi}\cdots\int_{0}^{\pi}
        \cos^{k}\phi_{1}\prod_{j=1}^{d-2}\sin^{n-1-j}\phi_{j}\,d\phi_{1}\ldots
        d\phi_{d-2}.
    \end{align*}
    To evaluate the iterated integrals we deduce
    from~\cite[3.621~(5)]{gradshteyn2007} that
    \[\int_{0}^{\pi}\sin^{a}x \cos^{b}x\,dx =
    \frac{1+(-1)^{b}}{2}B(\tfrac{a+1}{2},\tfrac{b+1}{2}),\]
    and therefore
    \begin{equation*}
        B_{k} = \frac{2\pi}{\vol(\SS)}
        \frac{\Gamma(\frac{k+1}{2})\Gamma(\frac{d-1}{2})}{\Gamma(\frac{k+d}{2})}
        \prod_{j=2}^{d-2}\frac{\Gamma(\frac{1}{2})\Gamma(\frac{d-j}{2})}{\Gamma(\frac{d-j+1}{2})}
        =\frac{2\pi^{(d-1)/2}}{\vol(\SS)}\frac{\Gamma(\frac{k+1}{2})}{\Gamma(\frac{k+d}{2})}.
    \end{equation*}
    Since $\vol(\SS)=2\pi^{d/2}/\Gamma(\tfrac{d}{2})$, we obtain the
    statement of the lemma.
\end{proof}

\subsection{Linear correlations}\label{sec:lin_corrs}

In this paper we mainly need to focus on the structure of $\Ccal(4)$
since bounds for longer correlations can be derived from it via a simple argument.
Let us denote by $\Dcal(\ell)=\Ccal(\ell)\setminus\Xcal(\ell)$ the
set of degenerate correlations in $\Ccal(\ell)$.
For $\Dcal(4)$, we further distinguish
between the symmetric correlations $\Dcal'$ defined as those
that cancel out in pairs, and the diagonal correlations $\Dcal''$ which are
multiples of a single pair of antipodal points $\pm\mu$.
Therefore, a summation over $\Ccal(4)$ can be decomposed as
\begin{equation}\label{eq:4decomp}
    \sum_{\Ccal(4)}
    =\sum_{\substack{\mu_1=-\mu_2\\\mu_3=-\mu_4}}
    +\sum_{\substack{\mu_1=-\mu_3\\\mu_2=-\mu_4}}
    + \sum_{\substack{\mu_1=-\mu_4\\\mu_2=-\mu_3}}
    + O\Bigl(\sum_{\Dcal''}+\sum_{\Xcal(4)}\Bigr),
\end{equation}
which we will use repeatedly in section~\ref{S:integrals}.
\begin{remark}
    In Proposition~\ref{intro:proposition} we show explicitly
    the dependence on $|\Xcal(4)|$
    so as to preserve the analogy with~\cite{krishnapur_nodal_2013}
    and~\cite{benatar_random_2017}.
    However, for the purpose of proving Theorem~\ref{intro:thm},
    the distinction between $|\Xcal(4)|$ and $|\Ccal(4)|$ is
    superfluous since we merely bound the former by the latter.
    This is different to the situation in two and three dimensions,
    where one can show that the majority of 4-correlations come from $\Dcal(4)$.
    When $d\geq 4$, the same argument fails due to the abundance of lattice
    points making it difficult to distinguish the non-degenerate correlations.
    On the other hand, because of the lower bound~\eqref{intro:lowerbounds}
    we do not lose much in using bounds only for $|\Ccal(4)|$.
\end{remark}

We have the following result to bound the total number of $\ell$-correlations.
\begin{lemma}\label{lemma:corrs}
Let $d\geq 4$ and $\ell\geq 4$. Then we have
\[
|\Ccal(\ell)| \ll \Ncal^{\ell-1-\alpha(d)+\epsilon},
\]
where $\alpha(d)$ is as in~\eqref{eq:alphad}.
\end{lemma}

\begin{proof}
    When $d\geq 5$, the result for $\ell=4$ follows immediately from~\cite[Theorem~2.7]{bourgain2015_annals},
    where we pick $n=d$ and $p=4$ in their theorem.
    Notice that in~\cite{bourgain2015_annals} the radius is $N$.
    The case $d=4$ and $\ell=4$ follows from~\cite[Theorem~3.1]{bourgain_new_2015}.
    For $\ell\geq 4$, we simply observe that
    \[|\Ccal(\ell)|=\int_{\TT^{d}}\Bigl\lvert
        \sum_{\mu\in\Ecal}e(\langle x,\mu\rangle)\Bigr\rvert^{\ell}dx\ll
        \Ncal^{\ell-4}\int_{\TT^{d}}\Bigl\lvert
        \sum_{\mu\in\Ecal}e(\langle x,\mu\rangle)\Bigr\rvert^{4}dx
    \ll \Ncal^{\ell-1-\alpha(d)+\epsilon},\]
    as required.
\end{proof}

%%%%%%%%%%%%%%%
%% SECTION 3 %%

\section{Proof of Proposition \ref{intro:proposition}}

In this section we prove Proposition \ref{intro:proposition}
under the assumption of two technical results that will be established
in sections \ref{S4} and \ref{S:integrals}.
Our argument follows a method successfully employed
in \cite{krishnapur_nodal_2013} and \cite{benatar_random_2017},
which itself builds on ideas from
\cite{oravecz_leray_2008} and \cite{rudnick_volume_2008}.
The starting point is a detailed analysis of the Kac--Rice formulas,
which allows us to express the variance $\Var(\Vcal)$
in terms of the covariance function
\begin{equation}\label{def:r}
r(x) = \frac{1}{\Ncal} \sum_{\mu\in\Ecal} e(\mu\cdot x)
\end{equation}
and its derivatives.
This is possible outside a \emph{singular} subset of the torus $\TT^d$,
whose measure is estimated again in terms of $r(x)$.
In particular, one requires the evaluation of the absolute moments
\begin{equation}\label{Rell}
    \Rcal(\ell) := \int_{\TT^d} |r(x)|^\ell\, dx = \frac{|\Ccal(\ell)|}{\Ncal^\ell},
\end{equation}
and of various sums over the sets of linear correlations $\Ccal(\ell)$,
with $\ell=2,4,6$.
Extracting the main term coming from the degenerate terms in these sums
and estimating the contribution of the non-degenerate ones
leads to the proof of Proposition \ref{intro:proposition}.
Notice that Lemma~\ref{lemma:corrs} immediately tells us that
\begin{equation}\label{eq:rlbound}
    \Rcal(\ell)\ll \Ncal^{-1-\alpha(d)+\epsilon}.
\end{equation}

\subsection{Kac--Rice formulas}
Let $x,y\in\TT^d$, and denote by $\Phi_{(F(x),F(y))}$ the joint density function
of the vector $(F(x),F(y))$. We work with the normalised \emph{two-point correlation}
(also called the ``second intensity'')
\[
    K_2(x) = \frac{d}{E} \cdot \Phi_{(F(x),F(y))}(0,0) \cdot \EE[\|\nabla F(y)\|\cdot\|\nabla F(x+y)\| : \; F(y)=F(x+y)=0].
\]
The Kac--Rice formulas (see e.g. \cite[\S6.2.1]{AWbook}) relate $\EE[\Vcal^2]$
to the function $K_2(x)$ by
\[
    \EE[\Vcal^2] = \frac{E}{d} \int_{\TT^d} K_2(x)\, dx.
\]
In particular, we have the following expression for the variance.
\begin{lemma}\label{lemmapage5}
    Let $\Gcal_d$ be as in \eqref{starstar}. Then we have
    \[
        \Var(\Vcal) = \frac{E}{d} \int_{\TT^d} \Bigl(K_2(x)-\frac{\Gcal_d^2}{4\pi^2}\Bigr) dx.
    \]
\end{lemma}
\begin{proof}
    Recall that $E=4\pi^2m$, and that from \eqref{starstar}
    we have $\EE[\Vcal]=\Gcal_d (m/d)^{1/2}$.
    The result then follows from the identity
    $\Var(\Vcal)=\EE[\Vcal^2]-\EE[\Vcal]^2$.
\end{proof}

Our goal is to express $K_2(x)$ in a more practical form
as a function of $r(x)$ and its derivatives.
To this purpose let us introduce the gradient and the Hessian of $r(x)$,
which we denote by $D$ and $H$, respectively. Thus we write
\begin{align}
    D(x) &:= \nabla r(x) =
    \frac{2\pi i}{\Ncal} \sum_{\mu\in\Ecal} e(\mu\cdot x) \mu,\label{def:D}\\
    H(x) &:= -\frac{4\pi^2}{\Ncal} \sum_{\mu\in\Ecal} e(\mu\cdot x) \;
    \mu^t\,\mu.\label{def:H}
\end{align}
Note that $D$ is a $d$-dimensional vector with entries
\[
    D_j(x) = \frac{2\pi i}{\Ncal} \sum_{\mu\in\Ecal} e(\mu\cdot x) \mu_j,
\]
and $H$ is a $d\times d$ matrix with entries
\[
    H_{jk}(x) = -\frac{4\pi^2}{\Ncal} \sum_{\mu\in\Ecal} e(\mu\cdot x) \mu_j\mu_k.
\]
Moreover, we define the quantities
\begin{equation}\label{def:XY}
    \begin{split}
        X(x) &= - \frac{d}{E} \frac{1}{1-r^2} D^t\cdot D,
        \\
        Y(x) &= - \frac{d}{E} \Bigl(H + \frac{r}{1-r^2} D^t\cdot D\Bigr),
    \end{split}
\end{equation}
and we let $\Omega$ be the $2d\times 2d$ block matrix
\begin{equation}\label{def:Omega}
    \Omega = \mathbb{I}_{2d} + \begin{pmatrix}X&Y\\ Y&X \end{pmatrix},
\end{equation}
where $\II_n$ denotes the identity matrix of dimension $n$.
With this notation we have the following result,
which gives a first simplification of $K_2$.

\begin{lemma}\label{lemmapage7}
    Let $X,Y,\Omega$ be as in \eqref{def:XY} and \eqref{def:Omega}. Then we have
    \begin{equation}\label{bullet}
        K_2(x) = \frac{1}{2\pi\sqrt{1-r^2(x)}} \; \EE[\|w_1\|\cdot\|w_2\|],
    \end{equation}
    where $w_1,w_2$ are $d$-dimensional random vectors with
    Gaussian distribution $(w_1,w_2)\sim \Ncal(0,\Omega)$.
\end{lemma}

\begin{proof}
    See \cite[Proposition 4.2]{benatar_random_2017}.
\end{proof}

\subsection{The singular set}\label{S3:singularset}
In order to further simplify $K_2$, we remove
a subset of the torus where the covariance function is large.
We show here that the contribution of this subset
will be absorbed into the error in Proposition~\ref{intro:proposition}.
The definitions and results in this subsection
are borrowed directly from \cite[\S6.1]{oravecz_leray_2008}
(see also \cite[\S4.1]{krishnapur_nodal_2013} and \cite[\S5.1]{benatar_random_2017}).

We say that a point $x\in\TT^d$ is \emph{positive singular} (respectively, \emph{negative singular})
if there exists a subset $\Ecal_x\subseteq\Ecal$ with density $|\Ecal_x|/|\Ecal|>1-\frac{1}{4d}$
such that $\cos(2\pi\mu\cdot x)>\frac{3}{4}$ (resp. $\cos(2\pi\mu\cdot x)<-\frac{3}{4}$),
for all $\mu\in\Ecal_x$.

Take $q\asymp \sqrt{m}$ and partition the torus into $q^d$ cubes,
each centred at $a/q$, $a\in\ZZ^d$, of side length $1/q$.
A cube $Q$ so constructed will be called positive singular (respectively, negative singular)
if it contains a positive (resp. negative) singular point.
We define the union of the singular cubes to be the singular set, and we denote it by $\Scal$.

\begin{lemma}\label{lemma:singular1}
    Let $\Scal\subseteq\TT^d$ denote the singular set introduced above.
    Then we have
    \[
        i)\; \int_{\Scal} K_2(x)\, dx \ll \meas(\Scal)
        \qquad\text{and}\qquad
        ii)\;\; \meas(\Scal) \leq (16)^\ell \, \Rcal(\ell),
    \]
    where $\ell$ is any non-negative integer and $\Rcal(\ell)$ is as in \eqref{Rell}.
\end{lemma}

\begin{proof}
    The first inequality was proved in \cite[\S6.3--\S6.5]{oravecz_leray_2008}.
    On the other hand we have $|r(x)|\geq 1/16$ for $x\in\Scal$
    by \cite[Lemma 6.5~$(ii)$]{oravecz_leray_2008}, from which the second estimate
    follows.
\end{proof}

Outside the singular set, the covariance function stays bounded away from one.
By \cite[Lemma 6.5$(i)$]{oravecz_leray_2008} we have, for all $x\notin\Scal$,
\begin{equation}\label{1902:eq001}
    |r(x)| \leq 1 - \frac{1}{16d}.
\end{equation}

\subsection{The non-singular set}
We now analyse the function $K_2(x)$ outside the singular set.
By \eqref{1902:eq001} the covariance function is
bounded away from one, and it is therefore possible to use the Taylor
expansion to approximate
\begin{equation}\label{1302:eq003}
    \frac{1}{\sqrt{1-r^2}} = 1 + \frac{1}{2}r^2 + \frac{3}{8} r^4 + O(r^6).
\end{equation}
As for the expectation $\EE[\|w_1\|\cdot \|w_2\|]$ in Lemma \ref{lemmapage7},
we prove the following explicit expression in terms of the functions
$X$ and $Y$ from \eqref{def:XY}. Since these are defined in terms of
$r,D$ and $H$, we indirectly arrive at a formula that involves only
the covariance function and its derivatives.

\begin{lemma}\label{1302:lemma001}
    Let $w_1,w_2$ be random vectors with Gaussian
    distribution $(w_1,w_2)\sim\Ncal(0,\Omega)$,
    where $\Omega$ is as in \eqref{def:Omega}. Then
    \begin{equation}\label{1702:eq001}
        \begin{split}
            \EE[\|w_1\|\cdot \|w_2\|]
            =
            \frac{\Gcal_d^2}{2\pi}
            \biggl[
            &A_0 + A_1\tr(X) + A_2\tr(Y^2) + A_3\tr(XY^2) + A_4\tr(X^2)
            \\
            +\,&
            A_5\tr(Y^4) + A_6\tr(Y^2)^2 + A_7\tr(X)\tr(Y^2)
            +
        O(\tr(X^3)+\tr(Y^6))\biggr],
        \end{split}
    \end{equation}
    where $X$ and $Y$ are given in \eqref{def:XY},
    $\Gcal_d$ is defined in \eqref{starstar}, and
    \begin{equation}\label{2102:eq006}
        \begin{gathered}
            A_0 = 1,
            \quad
            A_1 = \frac{1}{d},
            \quad
            A_2 = \frac{1}{2d^2},
            \quad
            A_3 = -\frac{1}{d^2(d+2)},
            \\
            A_4 = -\frac{d-1}{2d^2(d+2)},
            \quad
            A_5 = 2A_6 = \frac{1}{4d^2(d+2)^2},
            \quad
            A_7 = -\frac{1}{2d^2(d+2)}.
        \end{gathered}
    \end{equation}
\end{lemma}

Lemma \ref{1302:lemma001} is proved in the next section.
We assume it is true for now and proceed with the proof of
Proposition~\ref{intro:proposition}.
Since we can only expand $X$ and $Y$ outside the singular set, we need to be able
to extend our estimates to the full torus. This is possible due to the
following lemma which shows that $X$ and $Y$ are in fact uniformly bounded.
\begin{lemma}\label{1902:XYbouded}
    Let $X$ and $Y$ be given in \eqref{def:XY}. Then we have entrywise
    \[
        X,Y = O(1).
    \]
\end{lemma}
\begin{proof}
    This can be proved as in \cite[Lemma 3.2]{krishnapur_nodal_2013}:
    first observe that any entry of $X$ or $Y$ is dominated by the diagonal terms
    of the matrix $\Omega$ given in \eqref{def:Omega},  since $\Omega$ is a covariance matrix;
    next observe that the diagonal entries of $\Omega$ are non-negative;
    and finally deduce by \eqref{def:XY} that the diagonal entries of $X$ are uniformly bounded.
\end{proof}

We now have all the tools to express the second intensity $K_2(x)$
in terms of the covariance function $r(x)$ and the functions $X$ and $Y$.

\begin{proposition}\label{1902:proposition}
    Let $x\in\TT^d$ such that $x$ is not in the singular set $\Scal$. Then
    \[
        K_2(x) = \frac{\Gcal_d^2}{4\pi^2} + L(x) + \epsilon(x),
    \]
    where
    \[
        \begin{split}
            L(x)
            =
            \frac{\Gcal_d^2}{4\pi^2}
            &\biggl[
            \frac{A_0}{2}r^2 + A_1\tr(X) + A_2\tr(Y^2) + \frac{3}{8}r^4 + A_3\tr(XY^2) + A_4\tr(X^2)
            \\
            &+
            A_5\tr(Y^4) + A_6\tr(Y^2)^2 + A_7\tr(X)\tr(Y^2)
            +
            \frac{A_1}{2}r^2\tr(X) + \frac{A_2}{2}r^2\tr(Y^2)\biggr],
        \end{split}
    \]
    and
    \[
        \epsilon(x) = O(r^6 + \tr(X^3) + \tr(Y^6)).
    \]
\end{proposition}

\begin{proof}
    We insert \eqref{1302:eq003} and \eqref{1702:eq001} into~\eqref{bullet}
    and multiply out all the terms. After extracting the constant term and the term $L(x)$,
    we claim that the remaining terms are all absorbed into $\epsilon(x)$.
    This is readily verified by applying a combination of Young's inequality,
    the inequality $\tr(AB)\ll \tr(A^2)+\tr(B^2)$,
    the bound $|r(x)|\leq 1$,
    and the fact that $X,Y=O(1)$ by Lemma~\ref{1902:XYbouded}.
\end{proof}

With the help of Proposition~\ref{1902:proposition}
we can finally deduce Proposition~\ref{intro:proposition}.

\subsection{Proof of Proposition~\ref{intro:proposition}}
First, by Lemma~\ref{lemmapage5} we have
\[
\Var(\Vcal) = \frac{E}{d} \int_{\TT^d} \Bigl(K_2(x)-\frac{\Gcal_d^2}{4\pi^2}\Bigr) dx.
\]
We remove the singular set $\Scal$ and estimate its contribution with Lemma~\ref{lemma:singular1}
as
\[
\int_{\Scal} \Bigl(K_2(x)-\frac{\Gcal_d^2}{4\pi^2}\Bigr) dx \ll \Rcal(6).
\]
On the non-singular set $\TT^d\setminus\Scal$ we can expand $K_2(x)$ by
using Proposition~\ref{1902:proposition}, obtaining
\[
\Var(\Vcal) = \frac{E}{d} \int_{\TT^d\setminus\Scal} (L(x)+\epsilon(x))\, dx + O(\Rcal(6)).
\]
Next, recalling that $|r(x)|\leq 1$ and that $X$ and $Y$ are uniformly bounded by Lemma \ref{1902:XYbouded},
we can use again Lemma \ref{lemma:singular1}
to extend the integration to the whole torus $\TT^d$
at the cost of an error $O(\Rcal(6))$, arriving at the identity
\begin{equation}\label{1902:eq004}
\Var(\Vcal) = \frac{E}{d} \int_{\TT^d} L(x)\, dx + O\biggl(\Rcal(6) + \int_{\TT^d}(\tr(X^3)+\tr(Y^6))\, dx\biggr).
\end{equation}
We have collected in Lemma \ref{lemma:xyints}
the calculation of the integrals of the various terms
defining $L(x)$ and the error $\epsilon(x)$ over the full torus.
By points \eqref{int_X3} and \eqref{int_Y6} of Lemma \ref{lemma:xyints} the error contributes
\[
\int_{\TT^d}(\tr(X^3)+\tr(Y^6))\, dx \ll \frac{|\Ccal(6)|}{\Ncal^6} = \Rcal(6).
\]
The terms in $L(x)$ are evaluated in points \eqref{int_X}--\eqref{int_r2Y2} of Lemma \ref{lemma:xyints} and give
\begin{equation}\label{1902:eq003}
\begin{split}
\frac{E}{d} \int_{\TT^d} L(x)\, dx
=
\frac{m\Gcal_d^2}{d}
&\biggl[
\frac{1}{\Ncal}\biggl(\frac{1}{d}\cdot(-d)+\frac{1}{2d^2}\cdot(d^2)+\frac{1}{2}\biggr)
\\
&+
\frac{1}{\Ncal^2}\biggl(
- 1 - \frac{1}{d} + \frac{9}{8} + \frac{1}{d+2} - \frac{d-1}{2d} + \frac{d(2d+7)}{4(d+2)^3}
\\
&\phantom{xxxxxxxxx}
+
\frac{d(d^2+2d+6)}{8(d+2)^3}
+
\frac{d}{2(d+2)}
-
\frac{1}{2}
+
\frac{d+2}{4d}
\biggr)
\biggr],
\end{split}
\end{equation}
up to an error of size
\[
O\left(\frac{m}{\Ncal^2}\left(\frac{1}{m^{(d-3)/4-\epsilon}}+
\frac{|\Xcal(4)|}{\Ncal^2}+\frac{|\Ccal(6)|}{\Ncal^4}\right)\right).
\]
Simplifying the right-hand side of \eqref{1902:eq003} we see that
the coefficient of $1/\Ncal$ vanishes, which confirms that Berry's
cancellation phenomenon occurs for every $d$, as was mentioned in the introduction.
The remaining terms in \eqref{1902:eq003} give the identity
\[
\frac{E}{d} \int_{\TT^d} L(x)\, dx
=
\frac{m}{\Ncal^2}\left(
\frac{d-1}{d(d+2)^3} \Gcal_d^2
+
O\left(\frac{1}{m^{(d-3)/4-\epsilon}}+\frac{|\Xcal(4)|}{\Ncal^2}+
\frac{|\Ccal(6)|}{\Ncal^4}\right)
\right).
\]
Inserting this into \eqref{1902:eq004} we obtain Proposition \ref{intro:proposition}.
To complete the proof it remains therefore to prove Lemma \ref{1302:lemma001},
which will be done in section \ref{S4}, and to calculate the integral over $\TT^d$
of the various pieces definining $L(x)$ (and the error $\epsilon(x)$),
which will be done in section~\ref{S:integrals}.

%%%%%%%%%%%%%%%
%% SECTION 4 %%

\section{Proof of Lemma \ref{1302:lemma001}}\label{S4}

Recall that we are given random vectors $w_1,w_2$ with
joint distribution $(w_1,w_2)\sim\Ncal(0,\Omega)$,
where the covariance matrix $\Omega$ is given in \eqref{def:Omega}
(which in turn is defined in terms of $X$ and $Y$ from \eqref{def:XY}).
Our aim is to express the expectation $\EE[\|w_1\|\cdot\|w_2\|]$
in terms of $X$ and $Y$.
Our proof follows to a large extent those in
\cite[Lemma 5.1]{krishnapur_nodal_2013} and \cite[Lemma 5.8]{benatar_random_2017},
so we limit ourselves to outline the main steps and we direct the reader
to \cite{benatar_random_2017,krishnapur_nodal_2013} for further details.

First one can prove, following Berry \cite[(24)]{berry_statistics_2002}
(see also \cite[Lemma A.1]{benatar_random_2017}), that
\begin{equation}\label{2102:eq001}
\EE[\|w_1\|\cdot\|w_2\|] = \frac{1}{2\pi}\iint_{[0,\infty)^2}
\bigl(f(0,0)-f(t,0)-f(0,s)+f(t,s)\bigr)\frac{dt\, ds}{(ts)^{3/2}},
\end{equation}
where
\begin{equation}\label{2102:eq002}
f(t,s) = \frac{1}{\sqrt{\det(\mathbb{I}_{2d}+J(t,s))}}
\end{equation}
with
\begin{equation}\label{one}
\mathbb{I}_{2d} + J(t,s) =
\begin{pmatrix}
(1+t)\mathbb{I}_{d} + tX & \sqrt{ts}\,Y \\
\sqrt{ts}\,Y & (1+s)\mathbb{I}_{d}
\end{pmatrix}.
\end{equation}
Next, we expand the function $f(t,s)$ around $X=Y=0$.
By the formula for the determinant of a block matrix we can write
\[
\begin{split}
\det(\mathbb{I}_{2d}+J)
=&
\det((1+t)\mathbb{I}_d+tX)
\\
&\times
\det((1+s)\mathbb{I}_d+sX - \sqrt{ts} Y ((1+t)\mathbb{I}_d+tX)^{-1}\sqrt{ts}Y).
\end{split}
\]
We factor out $(1+t)$ and $(1+s)$ in the first and second term on the right,
respectively, and obtain
\begin{equation}\label{2102:eq003}
\begin{split}
\det(\mathbb{I}_{2d}+J)
=&
(1+t)^d \det\left(\mathbb{I}_d+\frac{t}{1+t}X\right)
\\
&\times
(1+s)^d \det\left(\mathbb{I}_d+\frac{s}{1+s}X - \frac{ts}{(1+t)(1+s)} Y \Bigl(\mathbb{I}_d+\frac{t}{1+t}X\Bigr)^{-1}Y\right).
\end{split}
\end{equation}
We insert \eqref{2102:eq003} in \eqref{2102:eq002} and use the Taylor approximations
\[(\II-P)^{-1} = \II + P + O(P^2),\]
and
\[\frac{1}{\sqrt{\det(\II+P)}} = 1 - \frac{1}{2}\tr(P) + \frac{1}{4} \tr(P^2) +
\frac{1}{8} \tr(P)^2 + O(\max_{i,j} |P_{ij}^3|),\]
as $P\to 0$.
We introduce the following shorthands
\[
\eta(t) = \frac{1}{(1+t)^{d/2}},
\qquad
\theta(t) = \frac{t}{(1+t)^{d/2+1}},
\qquad
\xi(t) = \frac{t^2}{(1+t)^{d/2+2}},
\]
and thus arrive at the identity
\begin{equation}\label{2102:eq004}
\begin{split}
f(t,s)
={}&
\eta(t)\eta(s)
-
\frac{1}{2}\theta(t)\theta(s)\tr(X)
+
\frac{1}{2}\theta(t)\theta(s)\tr(Y^2)
\\
&-
\frac{1}{2}(\theta(t)\xi(s)+\xi(t)\theta(s))\tr(XY^2)
+
\Bigl(\frac{3}{8}\xi(t)+\frac{3}{8}\xi(s)+\frac{1}{4}\theta(t)\theta(s)\Bigr)\tr(X^2)
\\
&+
\frac{1}{4}\xi(t)\xi(s)\tr(Y^4)
+
\frac{1}{8}\xi(t)\xi(s)\tr(Y^2)^2
-
\frac{1}{4}(\xi(t)\theta(s)+\theta(t)\xi(s))\tr(X)\tr(Y^2)
\\
&+
O\left(\left(\frac{ts}{(1+t)(1+s)}\right)(\tr(X^3)+\tr(Y^6))\right).
\end{split}
\end{equation}
Note that in the derivation of \eqref{2102:eq004}
one also uses the fact that $\tr(YXY)=\tr(XY^2)$ and $\tr(X)^2=\tr(X^2)$,
the latter due to $\rk(X)=1$.

Define $h(t,s)$ to be the integrand appearing in \eqref{2102:eq001}, that is,
\[
h(t,s) = f(0,0) - f(t,0) - f(0,s) + f(t,s).
\]
By \eqref{2102:eq004} we can expand $h(t,s)$ in the form
\begin{equation}\label{2102:eq005}
\begin{split}
h(t,s)
={}&
A_0(t,s)
+
A_1(t,s)\tr(X)
+
A_2(t,s)\tr(Y^2)
+
A_3(t,s)\tr(XY^2)
\\
&+
A_4(t,s)\tr(X^2)
+
A_5(t,s)\tr(Y^4)
+
A_6(t,s)\tr(Y^2)^2
+
A_7(t,s)\tr(X)\tr(Y^2)
\\
&+
O(\min(1,t)\min(1,s)(\tr(X^3)+\tr(Y^6)),
\end{split}
\end{equation}
where
\begin{align*}
    A_0(t,s) &= (1-\eta(t))(1-\eta(s)), & A_1(t,s) &= \frac{1}{2}(\eta(t)\theta(s)+\theta(t)\eta(s)),\\
    A_2(t,s) &= \frac{1}{2}\theta(t)\theta(s), & A_3(t,s) &=
    -\frac{1}{2}(\xi(t)\theta(s)+\theta(t)\xi(s)),\\
    %\shortintertext{\centering$\displaystyle A_4(t,s)=-\frac{3}{8}(\xi(t)(1-\eta(s))
    %    +\xi(s)(1-\eta(t)))+\frac{1}{4}\theta(t)\theta(s),$}
    A_4(t,s)&=\mathrlap{-\frac{3}{8}(\xi(t)(1-\eta(s))+\xi(s)(1-\eta(t)))+\frac{1}{4}\theta(t)\theta(s),}\\
    A_5(t,s) &= 2A_6(t,s) = \frac{1}{4}\xi(t)\xi(s), &
    A_7(t,s) &= -\frac{1}{4}(\theta(t)\xi(s)+\theta(s)\xi(t)).
\end{align*}
Finally, we multiply \eqref{2102:eq005} with $(ts)^{-3/2}$ and integrate
over both $t,s\in[0,\infty)$. Using the identities (see \cite[8.380.3]{gradshteyn2007})
\[
\int_0^\infty (1-\eta(t))\frac{dt}{t^{3/2}} = \Gcal_d,
\qquad
\int_0^\infty \theta(t)\frac{dt}{t^{3/2}} = \frac{\Gcal_d}{d},
\qquad
\int_0^\infty \xi(t)\frac{dt}{t^{3/2}} = \frac{\Gcal_d}{d(d+2)},
\]
we deduce that the desired expectation \eqref{2102:eq001} equals
\[
\begin{split}
\iint_{[0,\infty)^2} \!\!\! h(t,s) \frac{dt\, ds}{(ts)^{3/2}}
=
\frac{\Gcal_d^2}{2\pi}
&\biggl(
A_0 + A_1\tr(X) + A_2\tr(Y^2) + A_3\tr(XY^2)
+
A_4\tr(X^2)\\
&+
A_5\tr(Y^4) + A_6\tr(Y^2)^2 + A_7\tr(X)\tr(Y^2)\biggr)
+
O(\tr(X^3)+\tr(Y^6)),
\end{split}
\]
where $A_0,\ldots,A_7$ are given by \eqref{2102:eq006}.
This proves Lemma \ref{1302:lemma001}.

%%%%%%%%%%%%%%%
%% SECTION 5 %%

\section{Integrals}\label{S:integrals}

It remains to compute a number of integrals that we used in the expansion of
$K_{2}$ in~\eqref{1902:eq004} and~\eqref{1902:eq003}.
These should be compared to~\cite[Lemmas~5.10 and~6.1]{benatar_random_2017},
except that here we use the more general coefficients from Lemma~\ref{lemma:bk}.
In Lemma \ref{lemma:xyints} below
we let $\simf$ denote asymptotic up to an error of size
\[O\left(\frac{|\Xcal(4)|}{\Ncal^{4}}+\frac{|\Ccal(6)|}{\Ncal^{6}}\right),\]
and $\simfe$ denote asymptotic up to an error of size
\[O\left(\frac{1}{\Ncal^{2}m^{(d-3)/4-\epsilon}}+\frac{|\Xcal(4)|}{\Ncal^{4}}
+\frac{|\Ccal(6)|}{\Ncal^{6}}\right).\]

\begin{lemma}\label{lemma:xyints}
    Let $r$ be as in \eqref{def:r}
    and $X,Y$ be as in \eqref{def:XY}. We have
    \newline
    \noindent\begin{minipage}{.5\linewidth}
        \begin{align}
            &\int_{\TT^d} \tr(X) \,dx
            \simf-\frac{d}{\Ncal}-\frac{d}{\Ncal^{2}},        \label{int_X}\\
            &\int_{\TT^d} \tr(Y^2)\,dx
            \simf\frac{d^{2}}{\Ncal}-\frac{2d}{\Ncal^{2}},    \label{int_Y2}\\
            &\int_{\TT^d} \tr(XY^2)\,dx
            \simf-\frac{d^{2}}{\Ncal^{2}},                    \label{int_XY2}\\
            &\int_{\TT^d} \tr(X^2)\,dx
            \simf\frac{d(d+2)}{\Ncal^{2}},                    \label{int_X2}\\
            &\int_{\TT^d} \tr(Y^4)\,dx
            \simfe \frac{d^{3}(2d+7)}{(d+2)}\cdot\frac{1}{\Ncal^{2}}, \label{int_Y4}\\
            &\int_{\TT^d} \tr(Y^2)^2\,dx
            \simfe \frac{d^{3}(d^{2}+2d+6)}{(d+2)}\cdot\frac{1}{\Ncal^{2}}, \label{int_Y22}
        \end{align}
    \end{minipage}%
    \begin{minipage}{.5\linewidth}
        \begin{align}
            &\int_{\TT^d} \tr(X)\tr(Y^2)\,dx
            \simf-\frac{d^{3}}{\Ncal^{2}},                    \label{int_XtrY2}\\
            &\int_{\TT^d} r^2\tr(X)\,dx
            \simf-\frac{d}{\Ncal^{2}},                        \label{int_r2X}\\
            &\int_{\TT^d} r^2\tr(Y^2)\,dx
            \simf\frac{d(d+2)}{\Ncal^{2}},                    \label{int_r2Y2}\\
            &\int_{\TT^d} \tr(X^3)\,dx
            \ll\frac{|\Ccal(6)|}{\Ncal^{6}},                 \label{int_X3}\\
            &\int_{\TT^d} \tr(Y^6)\,dx
            \ll\frac{|\Ccal(6)|}{\Ncal^{6}}.                \label{int_Y6}
        \end{align}
    \end{minipage}
\end{lemma}
\begin{proof}
    First, for each of the integrals we bound the integrand separately in the
    singular and non-singular set. Therefore, by Lemma~\ref{1902:XYbouded}
    we have, say,
    \[\int_{\TT^{d}} \tr(X)\,dx = \int_{\TT^{d}\setminus\Scal} \tr(X)\,dx
    +O(\meas(\Scal)).\]
    We bound the error by Lemma~\ref{lemma:singular1} and in the integral
    use the definition of $X$~\eqref{def:XY}. Since we are not on the singular
    set, we can Taylor expand the $(1-r^{2})^{-1}$ factor and get
    \begin{multline*}
      \int_{\TT^{d}} \tr(X)\,dx =
      -\frac{d}{E}\left(\int_{\TT^{d}}(DD^{t})\,dx
      +\int_{\TT^{d}}r^{2}(DD^{t})\,dx\right)
      \\+O\left(\frac{1}{E}\int_{\TT^{d}}r^{4}(DD^{t})\,dx\right)+
      O\left(\frac{|\Ccal(6)|}{\Ncal^{6}}\right) + O(\Rcal(\ell)),
    \end{multline*}
    where we have reintroduced the contribution over the singular set
    via an error that is absorbed in $O(\Rcal(\ell))$. The term
    $O(\Rcal(\ell))$ is admissible by~\eqref{eq:rlbound} so we will not write it in
    the analysis below.
    We now apply Lemma~\ref{lemma:dhrints} and in particular
    from~\eqref{int_dd}, \eqref{int_r2dd} and~\eqref{int_r4dd} we
    deduce that
    \[\int_{\TT^{d}}\tr(X)\,dx\simf -\frac{d}{\Ncal}-\frac{d}{\Ncal^{2}}.\]
    The remaining integrals in the lemma are bounded
    with an identical argument except for the final two
    integrals~\eqref{int_X3} and~\eqref{int_Y6}. For these we bound
    trivially by using~\eqref{int_dd3} and~\eqref{int_h6}, respectively.
\end{proof}

Let $\simD$ denote asymptotic up to an error of size
\[O\left(\frac{1}{\Ncal^{3}}+\frac{|\Xcal(4)|}{\Ncal^{4}}\right),\]
and $\sime$ denote asymptotic up to an error of size
\[O\left(\frac{1}{\Ncal^{2}m^{(d-3)/4-\epsilon}}+\frac{|\Xcal(4)|}{\Ncal^{4}}
\right).\]
Notice that $\Ncal^{-1}\ll m^{-(d-3)/4+\epsilon}$ by~\eqref{eq:Nsize}.

\begin{lemma}\label{lemma:dhrints}
    Let $r$ be as in \eqref{def:r}
    and $D,H$ be as in \eqref{def:D},\eqref{def:H}. We have
    \newline
    \noindent\begin{minipage}{.46\linewidth}
        \begin{align}
            &\int_{\TT^d} r^2\, dx
            = \frac{1}{\Ncal},                                   \label{int_r2}\\
            &\int_{\TT^d} r^4 \,dx
            \simD \frac{3}{\Ncal^{2}},                            \label{int_r4}\\
            \frac{1}{E}&\int_{\TT^d} (D D^t) \,dx
            =\frac{1}{\Ncal},                                    \label{int_dd}\\
            \frac{1}{E^2}&\int_{\TT^d} (D D^t)^2 \,dx
            \simD\frac{d+2}{d}\cdot\frac{1}{\Ncal^{2}},                \label{int_dd2}\\
            \frac{1}{E}&\int_{\TT^d} (r^2 D D^t) \,dx
            \simD\frac{1}{\Ncal^{2}},                             \label{int_r2dd}\\
            \frac{1}{E^2}&\int_{\TT^d} \tr(H^2)\,dx
            =\frac{1}{\Ncal},                                    \label{int_h2}\\
            \frac{1}{E^2}&\int_{\TT^d} (r^2 \tr(H^2))\,dx
            \simD\frac{d+2}{d}\cdot\frac{1}{\Ncal^{2}},                \label{int_r2h2}\\
            \frac{1}{E^4}&\int_{\TT^d} \tr(H^4)\,dx
            \sime\frac{2d+7}{d(d+2)}\cdot\frac{1}{\Ncal^{2}},          \label{int_h4}
        \end{align}
    \end{minipage}%
    \begin{minipage}{.54\linewidth}
        \begin{align}
            \frac{1}{E^4}&\int_{\TT^d} \tr(H^2)^2 \,dx
            \sime\frac{d^{2}+2d+6}{d(d+2)}\cdot\frac{1}{\Ncal^{2}},    \label{int_h22}\\
            \frac{1}{E^3}&\int_{\TT^d} (D D^t \tr(H^2))\,dx
            \simD\frac{1}{\Ncal^{2}},                         \label{int_ddh2}\\
            \frac{1}{E^2}&\int_{\TT^d} (r D H D^t) \,dx
            \simD-\frac{1}{d}\cdot\frac{1}{\Ncal^{2}},             \label{int_rdhd}\\
            \frac{1}{E^3}&\int_{\TT^d} (D H^2 D^t) \,dx
            \simD\frac{1}{d}\cdot\frac{1}{\Ncal^{2}},              \label{int_dh2d}\\
            \frac{1}{E^3}&\int_{\TT^d} (D D^t)^3 \,dx
            \ll \frac{|\Ccal(6)|}{\Ncal^{6}},                \label{int_dd3}\\
            \frac{1}{E}&\int_{\TT^d} (r^4 D D^t)\,dx
            \ll \frac{|\Ccal(6)|}{\Ncal^{6}},                \label{int_r4dd}\\
            \frac{1}{E^6}&\int_{\TT^d} \tr(H^6)\,dx
            \ll \frac{|\Ccal(6)|}{\Ncal^{6}},                \label{int_h6}\\
            \frac{1}{E^3}&\int_{\TT^d} (r D D^t D H D^t)\,dx
            \ll \frac{|\Ccal(6)|}{\Ncal^{6}}.               \label{int_rdddhd}
        \end{align}
    \end{minipage}
\end{lemma}

\begin{proof}
    The integral~\eqref{int_r2} follows from the definition of $r$~\eqref{def:r}:
    \[\int_{\TT^{d}}
        r^{2}\,dx=\frac{1}{\Ncal^{2}}\sum_{\mu_{1},\mu_{2}}\int_{\TT^{d}}e(\langle\mu_{1}+\mu_{2},x\rangle)\,dx
    =\frac{1}{\Ncal^{2}}\sum_{\mu_{1}+\mu_{2}=0}1=\frac{1}{\Ncal}.\]
    For~\eqref{int_r4} we use the decomposition~\eqref{eq:4decomp} of
    4-correlations to get
    \[\int_{\TT^{d}}r^{4}\,dx=\frac{1}{\Ncal^{4}}\sum_{\mu_{1}+\ldots+
            \mu_{4}=0}1=\frac{3}{\Ncal^{2}}+O\left(\frac{1}{\Ncal^{3}}+
    \frac{|\Xcal(4)|}{\Ncal^{4}}\right).\]
    The remaining integrals are treated in a similar fashion, but we have to be
    more careful since the summands are more complicated inner products.
    Now,
    \begin{align*}
        \int_{\TT^{d}} (DD^{t}) \,dx
            &= \frac{-(2\pi)^{2}}{\Ncal^{2}}\sum_{\mu_{1},\mu_{2}}\int_{\TT^{d}}
            e(\langle \mu_{1}+\mu_{2},x\rangle)\mu_{1}\mu_{2}^{t}\,dx\\
            &=\frac{-4\pi^{2}}{\Ncal^{2}}\sum_{\mu_{1}+\mu_{2}=0}\mu_{1}\cdot\mu_{2}
            %\\&
            =\frac{-4\pi^{2}}{\Ncal^{2}}\sum_{\mu_{1}}-|\mu_{1}|^{2}=\frac{E}{\Ncal}.
    \end{align*}
    In the rest of the integrals we will repeatedly express the sums in terms
    of the numbers $B_{k}$ (see~\eqref{Bk:def}). First,
    \begin{align*}
        \int_{\TT^{d}}(DD^{t})^{2}\,dx
        &= \frac{(4\pi^{2})^{2}}{\Ncal^{4}}\sum_{\Ccal(4)}
        \mu_{1}\mu_{2}^{t}\mu_{3}\mu_{4}^{t}\\
        &\simD\frac{(4\pi^{2})^{2}}{\Ncal^{4}}\Biggl(\sum_{\substack{\mu_{1}+\mu_{2}=0\\
            \mu_{3}+\mu_{4}=0}}(-\mu_{1}\cdot\mu_{1})(-\mu_{3}\cdot\mu_{3})
            +2\sum_{\substack{\mu_{1}+\mu_{3}=0\\
        \mu_{2}+\mu_{4}=0}}(\mu_{1}\cdot\mu_{2})^{2}\Biggr),
    \end{align*}
    by symmetry. Thus,
    \[\int_{\TT^{d}}(DD^{t})^{2}\,dx
        \simD\frac{(4\pi^{2})^{2}m^{2}}{\Ncal^{2}}(1+2B_{2})\simD
    \frac{E^{2}}{\Ncal^{2}}\frac{d+2}{d}.\]
    Similarly,
    \[\int_{\TT^{d}}(r^{2}DD^{t})\,dx = \frac{-4\pi^{2}}{\Ncal^{4}}\sum_{\Ccal(4)}
        \mu_{3}\mu_{4}^{t}\simD\frac{-4\pi^{2}}{\Ncal^{4}}(-m\Ncal^{2}+2m\Ncal^{2}B_{1}),\]
    since only the degenerate correlations with $\mu_{3}=-\mu_{4}$ contribute.

    Next we have integrals that contain the trace of the Hessian. From the
    definition of $H$~\eqref{def:H} it easily follows in~\eqref{int_h2} that
    \[\int_{\TT^{d}}\tr(H^{2})\,dx =
        \frac{(4\pi^{2})^{2}}{\Ncal^{2}}\sum_{\mu_{1}+\mu_{2}=0}\tr(\mu_{1}^{t}
        \mu_{1}\mu_{2}^{t}\mu_{2})=\frac{(4\pi^{2})^{2}}{\Ncal^{2}}\sum_{\mu_{1}}
    (\mu_{1}\cdot\mu_{1})^{2}=\frac{E^{2}}{\Ncal},\]
    since the $\mu_{i}$ are $1\times d$ row vectors.
    For~\eqref{int_r2h2} we then obtain
    \begin{equation*}
        \int_{\TT^{d}}r^{2}\tr(H^{2})\,dx
        = \frac{(4\pi^{2})^{2}}{\Ncal^{4}}
        \sum_{\Ccal(4)}\tr(\mu_{3}^{t}\mu_{3}\mu_{4}^{t}
        \mu_{4})
        \simD\frac{(4\pi^{2})^{2}}{\Ncal^{4}}(m^{2}\Ncal^{2}
        +2m^{2}\Ncal^{2}B_{2})
        \simD\frac{E^{2}}{\Ncal^{2}}\frac{d+2}{d},
    \end{equation*}
    as required. Following the same train of thought, we get for~\eqref{int_h4}
    that
    \[\int_{\TT^{d}}\tr(H^{4})\,dx =
        \frac{E^{4}}{\Ncal^{2}}(2B_{2}+B_{4})+O\left(\frac{1}{\Ncal^{3}}+
        \frac{|\Xcal(4)|}{\Ncal^{4}}\right)
    \sime\frac{E^{4}}{\Ncal^{2}}\left(\frac{2}{d}+\frac{3}{(2+d)d}\right),\]
    where the additional error is introduced by the $B_{4}$ term and
    Lemma~\ref{lemma:bk}.
    Analogously for~\eqref{int_h22} we get
    \[\int_{\TT^{d}}\tr(H^{2})^{2}\,dx = \frac{E^{4}}{\Ncal^{2}}(1+2B_{4})+
        O\left(\frac{1}{\Ncal^{3}}+\frac{|\Xcal(4)|}{\Ncal^{4}}\right)
    \sime\frac{E^{4}}{\Ncal^{2}}\frac{d^{2}+2d+6}{d(d+2)}.\]
    Now,
    \[\int_{\TT^{d}}(DD^{t}\tr(H^{2}))\,dx = \frac{-(4\pi^{2})^{3}}{\Ncal^{4}}
        \sum_{\Ccal(4)}\mu_{1}\mu_{2}^{t}\tr(\mu_{3}^{t}\mu_{3}\mu_{4}^{t}\mu_{4})
    \simD\frac{-(4\pi^{2})^{3}}{\Ncal^{4}}(-m^{3}\Ncal^{2}+2B_{3}),\]
    and similarly
    \[\int_{\TT^{d}}(rDHD^{t})\,dx =
        \frac{E^{2}}{\Ncal^{2}}(2B_{1}-B_{2})\simD\frac{-E^2}{d}\cdot\frac{1}{\Ncal^{2}}.
    \]
    Furthemore,
    \[\int_{\TT^{d}}(DH^{2}D^{t})\,dx \simD
    \frac{-E^{3}}{\Ncal^{2}}(B_{1}+B_{3}-B_{2}),\]
    which simplifies to the correct value since $B_{1}=B_{3}=0$.

    For the remaining integrals~\eqref{int_dd3}--\eqref{int_rdddhd} we bound
    everything trivially, e.g.
    \[\int_{\TT^{d}}(DD^{t})^{3}\,dx = \frac{-(4\pi^{2})^{3}}{\Ncal^{6}}
        \sum_{\Ccal(6)}\mu_{1}\mu_{2}^{t}\mu_{3}\mu_{4}^{t}
        \mu_{5}\mu_{6}^{t}\ll \frac{(4\pi^{2})^{3}}{\Ncal^{6}}(\sqrt{m})^{6}
    |\Ccal(6)|,\]
    and similarly for the rest.
\end{proof}

\end{document}